\title{Non-existence of stable solutions for weighted $p$-Laplace equation}
\author{Kaushik Bal
        \and
        Prashanta Garain
        }
\documentclass{article}

\usepackage{graphicx,fancyhdr}
\usepackage{hyperref,amsfonts,amsmath,amsthm,bigints}
\usepackage{enumitem,bigints}
\usepackage[usenames, dvipsnames]{color}
\newtheorem{The}{Theorem}[section]

\newtheorem{Rem}{Remark}[section]

\newtheorem{Def}{Definition}[section]

\newenvironment{AMS}{}{}
\newenvironment{keywords}{}{}

\begin{document}
\newpage
\maketitle
\begin{abstract}
We provide sufficient conditions on $w\in L^1_{loc}(\mathbb{R}^N)$ such that the weighted $p$-Laplace equation $$-\operatorname{div}\big(w(x)|\nabla u|^{p-2}\nabla u\big)=f(u)\;\;\mbox{in}\;\;\mathbb{R}^N$$ does not admit any stable $C^{1,\zeta}_{loc}$ solution in $\mathbb{R}^N$ where $f(x)$ is either $-x^{-\delta}$ or $e^x$ for any $0<\zeta<1$.
\end{abstract}

\begin{keywords}
\textbf{Keywords:} p-Laplacian; Non-existence; Stable solution
\end{keywords}

\begin{AMS}
\textbf{2010 Subject Classification} 35A01, 35B93, 35J92
\end{AMS}

\section{Introduction}
In this paper we are interested in finding conditions under which the quasilinear degenerate equations
given by $$-\Delta_{p,w} u=-u^{-\delta}\;\mbox{in}\;\;\mathbb R^N\;\;\mbox{and}\;\;-\Delta_{p,w} u=e^u\;\mbox{in}\;\;\mathbb R^N$$
doesn't admit a stable solution in $\mathbb{R}^N$.\\
Here $-\Delta_{p,w} u:=\operatorname{div}\big(w(x)|\nabla u|^{p-2}\nabla u\big)$ where $w:\mathbb{R}^N\to \mathbb{R}$ is a positive, measurable function satisfies one of the following conditions:
\begin{itemize}
\item $w\in C^1(\mathbb{R}^N)$ such that $w(x)\geq \eta$ for some constant $\eta>0$.
\item There exists $M>0$ such that $w(x)=w(|x|)=w(r)\leq M|r|^{\theta}$ for $\theta\in\mathbb{R}^N$ and $r=|x|$.
\end{itemize}  
The study of stable solutions for elliptic equation has been a subject of interest for the last three decades. When $p=2$ and $w\equiv 1$ the general problem of stable solutions of the equation given by $$-\Delta u=f(u)\;\;\mbox{in}\;\;\Omega\label{eo}$$ with $f$ locally lipschitz continuous in $\mathbb R$ has been a subject of considerable research. Obtaining $L^{\infty}$ estimates is one of the major concern in this case and we provide here a brief description of the available results.
\begin{itemize}
\item For the case $3\leq N\leq 9$, Cabr\'e and Capella \cite{Ca1} settled the problem for the unit ball.
\item a general smooth bounded domain the boundedness was derived for $N\leq 3$ and $N=4$ respectively by Nedev \cite{Ne} and Cabr\'e \cite{Ca} respectively
\item When $N\geq 10$ the existence of unbounded solutions was shown in Cabr\'e and Capella \cite{Ca1}.
\end{itemize}   
For more information on this field one can consult the survey of Cabr\'e \cite{Ca3}. In case of Laplacian in $\mathbb{R}^N$, classification results for $f(u)=|u|^{p-1}u$, $p>1$ or when $f(u)=e^u$ are already available see Farina \cite{Fa, Fa1}. In Farina \cite{Fa}, non-existence of stable solutions for $2\leq N\leq 9$ was obtained when $f(u)=e^u$. Other significant work in these topic can be found in \cite{JWei} and reference therein.
For the equation 
\begin{equation}\label{mp}
-\operatorname{div}\big(w(x)|\nabla u|^{p-2}\nabla u\big)=g(x)f(u)\;\;\mbox{in}\;\;\mathbb R^N
\end{equation}
 the following results were obtained recently
\begin{itemize}
\item When $w=g=1$, Le \cite{Ph} showed non-existence result for $p>2$ and $N<\frac{p(p+3)}{p-1}$ when $f(x)=e^x$.
\item When $w=g=1$, Guo-Mei \cite{GuoMei} showed non-existence for finite Morse index solutions for $2\leq p<N<\frac{p(p+3)}{p-1}$ when $f(x)=-x^{-\delta},\;\delta>q_c$ and $q_c$ is given by 
$$q_c=\frac{(p-1)[(1-p)N^2+(p^2+2p)N-p^2]-2p^2[(p-1)(N-1)]^{\frac{1}{2}}}{(N-p)[(p-1)N-p(p+3)]}$$
\item 
For $g\in L^1_{loc}(\mathbb{R}^N)$ with $g(x)\geq C|x|^a$ for $|x|$ large enough, Chen et al \cite{ChSoYa} showed the non-existence results for the case $f(x)=-x^{-\delta},\;\delta>q_c$ and
$2\leq p<N<\mu_0(p,a):=\frac{p(p+3)+4a}{p-1}$ and $$q_c=\frac{2(N+a)(p+a)-(N-p)[(p-1)(N+a)-p-a]-\alpha}{(N-p)[N(p-1)-p(p+3)]}$$ where $\alpha=2(p+a)\sqrt{(p+a)\big(N+a+\frac{N-p}{p-1}\big)}$.\\
For $f(x)=e^x$ they obtained that for $1\leq p<N<\frac{p(p+3)+4a}{p-1}$ the problem does not have a stable solution.
\end{itemize}

Our main aim in this note is to establish some non-existence results for stable solutions of the equation (\ref{mp}) provided $g(x)\equiv 1$. Before we begin with the main results let us define the notion of weak solution and stable solution for the problem (\ref{mp}).
\begin{Def}
We say that $u\in C^{1,\zeta}_{loc}(\mathbb{R}^N)$ is a weak solution to the equation $(\ref{mp})$ if 
for all $\varphi\in C_{c}^{1}(\mathbb{R}^N)$ we have,
\begin{equation}\label{ws1}
\int_{\mathbb{R}^N}w(x)|\nabla u|^{p-2} \nabla u\nabla \varphi dx-\int_{\mathbb{R}^N}f(u)\varphi dx=0.
\end{equation}
\end{Def}
\begin{Def}
A weak solution $u$ of equation $(\ref{mp})$ is said to be stable if for all $\varphi\in C_{c}^{1}(\mathbb R^N)$ we have,
\begin{multline}\label{ws}
\int_{\mathbf{R}^N}w(x)|\nabla u|^{p-2}|\nabla \varphi|^2 dx+(p-2)\int_{\mathbf{R}^N}w(x)|\nabla u|^{p-4}(\nabla u,\nabla \varphi)^2 dx\\-\int_{\mathbf{R}^N}f'(u)\varphi^{2}dx\geq 0.
\end{multline}
Therefore if $u$ is a stable solution of equation $(\ref{mp})$ then we have,
\begin{equation}\label{ss}
\int_{\mathbf{R}^N}f'(u)\varphi^{2}dx\leq(p-1)\int_{\mathbb{R}^N}w(x)|\nabla u|^{p-2}|\nabla \varphi|^2 dx.
\end{equation}
\end{Def}

\section{Main Results}
We start by denoting the equations as $(n)_e$ and $(n)_s$ for $f(x)=e^x$ and $f(x)=-x^{-\delta}$ respectively with $n=1,2,3,4$ and $$A_r=||w||^{\frac{p+3}{p-1}}_{L^{\frac{p+3}{p-1}}(B_{2r}(0))}$$ where $B_{r}(0)$ is the ball centered at $0$ with radius $r\geq 0$. We will also assume $C>0$ to be a constant for the rest of the paper which may vary depending on the situation.
\begin{The}\label{nonexistence theorem}
Let $p>2$ and $w\in C^1(\mathbb{R}^N)$ be such that $w(x)>C$. If for every ball  $B_R(0)$ we have, $A_R=\mathcal{O}(R^\mu)$ where $\mu<\frac{p(p+3)}{p-1}$, then there does not exist any stable solution of the problem $(\ref{mp})_e$ in $C^{1,\zeta}_{loc}(\mathbb{R}^N)$ for any $0<\zeta<1$. 
\end{The}
\begin{The}\label{radial}
Let $p>2$ and $w\in L^{\frac{p+3}{p-1}}_{loc}(\mathbb{R}^N)$ be such that $w(x)=w(|x|)\leq C|x|^{\theta}$, then the problem $(\ref{mp})_e$  does not admit a stable solution in $C^{1,\zeta}_{loc}(\mathbb{R}^N\setminus \{0\})$ for any $0<\zeta<1$ provided $\theta\in\cup_{\alpha}\mathcal{I}_{\alpha}$, where $\mathcal{I}_{\alpha}=(-\frac{N}{p\alpha+1},p-\frac{N}{p\alpha+1})$ for $\alpha\in (0,\frac{4}{p(p-1)})$. 
Moreover
\begin{enumerate}[label=\roman*)]
\item If $N<\frac{p(p+3)}{p-1}$, there exists $\alpha_0\in (0,\frac{4}{p(p-1)})$ such that 
$0\in\mathcal{I}_{\alpha_0}$.
\item If $N\geq\frac{p(p+3)}{p-1}$, then $\cup_{\alpha}\mathcal{I}_{\alpha}\subset (-\infty,0).$
\end{enumerate}
\end{The}

\begin{Rem} For $w(x)\equiv 1$ the equation $(\ref{mp})_e$ becomes the following equation 
\begin{equation}\label{plap1}
-\Delta_p u=e^u\;\mbox{in}\;\;\mathbb R^N
\end{equation}
and Theorem \ref{radial} says that there does not exists any stable solution to the equation (\ref{plap1}) provided one has $N<\frac{p(p+3)}{p-1}$, $p>2$. 
\end{Rem}

For $\delta>0$ denote, $$k=\frac{\delta+p-1+2l}{\delta+p-1}\;\mbox{where}\;l=\frac{2\delta}{p-1}-\frac{p-1}{2} \;\mbox{and}\;Q_r=||w||^k_{L^k(B_{2r}(0))}$$ Also $$\mathcal{B}:=(0,l)$$

Note that for $\delta>0$ and $p>2$ we have $k>1$.
\begin{The}\label{Main}
Suppose $w\in C^1(\mathbb{R}^N)$ be such that $w(x)>C$ and for every ball  $B_{2R}(0)$ we have, $Q_R=\mathcal{O}(R^\tau)$. If $0<\tau<pk$ with $p>2$ and $\delta\geq{max\{1,\frac{(p-1)^2}{4}}\}$ then there is no stable, positive solution of equation $(\ref{mp})_s$ in $C^{1,\zeta}_{loc}(\mathbb{R}^N)$
for any $0<\zeta<1$.
\end{The}

\begin{The}\label{rad}
Let $p>2$ and $w\in L^k_{loc}(\mathbb{R}^N)$ be such that $w(x)=w(|x|)\leq C |x|^{\theta}$. If $\theta\in \cup_{\beta}\mathcal{K}_{\beta}$ where $\mathcal{K}_{\beta}=(-\frac{N(\delta+p-1)}{2\beta+\delta+p-1},p-\frac{N(\delta+p-1)}{2\beta+\delta+p-1})$ for $\beta\in \mathcal{B}$ then there does not exist any positive, stable solution of the problem $(\ref{mp})_s$ in $C^{1,\zeta}_{loc}(\mathbb{R}^N\setminus \{0\})$ for any $0<\zeta<1$.
Moreover,
\begin{enumerate}[label=\roman*)]
\item If $N<kp$ then there exist $\beta_0\in \mathcal{B}$ such that $0\in \mathcal{K}_{\beta_0}$
\item If $N\geq kp$ then $\cup_{\beta\in\mathcal{B}}\mathcal{K}_{\beta}\subset (-\infty,0)$
\end{enumerate}
\end{The}

\begin{Rem} For $w(x)\equiv 1$ the equation $(\ref{mp})_s$ becomes the following equation 
\begin{equation}\label{plap}
-\Delta_p u+u^{-\delta}=0\;\mbox{in}\;\;\mathbb R^N
\end{equation}
and Theorem \ref{rad} says that there does not exists any positive, stable solution to the equation (\ref{plap}) provided one has $N<kp$, $p>2$ and $\delta\geq{max\{1,\frac{(p-1)^2}{4}}\}.$ 
\end{Rem}

\section{Proof of Main Theorems}
\begin{proof}[Proof of Theorem \ref{nonexistence theorem}]
Suppose $u \in C^{1,\zeta}(\mathbb R^N)$, $0<\zeta<1$ be a stable solution of equation $(\ref{mp})_e$ for $\mu<\frac{p(p+3)}{p-1}$ and let $\psi\in C_c^{1}(\mathbb{R}^N)$. \\
\textbf{Step 1.}
Choosing $\varphi=e^{p\alpha u}\psi^p$ ($\alpha>0$ to be chosen later) in equation $(\ref{ws1})_e$.
Since, $$\nabla\varphi = p\alpha e^{p\alpha u}\psi^p\nabla u + pe^{p\alpha u}\psi^{p-1}\nabla\psi$$ 
we obtain
\begin{multline*}
p\alpha\int_{\mathbf{R}^N} w(x)|\nabla u|^p e^{p\alpha u} \psi^p \,dx 
+ p\int_{\mathbf{R}^N} w(x)|\nabla u|^{p-2} e^{p\alpha u}
 \psi^{p-1} <\nabla u, \nabla\psi> \,dx \\
= \int_{\mathbf{R}^N} e^{(p\alpha + 1) u} \psi^p \,dx.
\end{multline*}
Therefore using Young's inequality for any $\epsilon\in(0,p\alpha)$, we obtain
\begin{multline*}
p\alpha\int_{\mathbf{R}^N} w(x)|\nabla u|^p e^{p\alpha u} \psi^p \,dx\\ 
\leq p\int_{\mathbf{R}^N} w(x)|\nabla u|^{p-1} e^{p\alpha u} \psi^{p-1} |\nabla\psi| \,dx
  + \int_{\mathbf{R}^N} e^{(p\alpha + 1) u} \psi^p \,dx\\
\leq \int_{\mathbf{R}^N} \epsilon
 \Big(w^{\frac{p-1}{p}}|\nabla u|^{p-1} e^{(p-1)\alpha u} \psi^{p-1} \Big)^{\frac{p}{p-1}}\\
 + C_\epsilon\big(w^{\frac{1}{p}}e^{\alpha u} |\nabla\psi| \big)^p \,dx
 + \int_{\mathbf{R}^N} e^{(p\alpha + 1) u} \psi^p \,dx\\
= \epsilon\int_{\mathbf{R}^N} w(x)|\nabla u|^p e^{p\alpha u} \psi^p \,dx 
 + C_\epsilon\int_{\mathbf{R}^N} w(x)e^{p\alpha u} |\nabla\psi|^p \,dx \\
 + \int_{\mathbf{R}^N} e^{(p\alpha + 1) u} \psi^p \,dx
\end{multline*}
Therefore we get,
\begin{multline}\label{s1}
(p\alpha-\epsilon)\int_{\mathbf{R}^N} w(x)|\nabla u|^p e^{p\alpha u} \psi^p \,dx 
\leq  C_\epsilon\int_{\mathbf{R}^N} w(x)e^{p\alpha u} |\nabla\psi|^p \,dx \\
+ \int_{\mathbf{R}^N} e^{(p\alpha + 1) u} \psi^p \,dx.
\end{multline}
\textbf{Step 2.}
Choose $\varphi = e^{\frac{p\alpha u}{2}} \psi^{\frac{p}{2}}$. Therefore,
$$
\nabla\varphi = \frac{p\alpha}{2} e^{\frac{p\alpha u}{2}} \psi^{\frac{p}{2}} \nabla u 
+ \frac{p}{2}e^{\frac{p\alpha u}{2}} \psi^{\frac{p-2}{2}} \nabla\psi
$$
Putting $\varphi$ and $\nabla\varphi$ in the stability equation $(\ref{ss}_e)$, we obtain
\begin{align*}\label{s2.1}
\int_{\mathbf{R}^N} e^{(p\alpha + 1) u} \psi^p \,dx 
&\leq (p-1)\int_{\mathbf{R}^N} w(x)|\nabla u|^p \left(\frac{p\alpha}{2}\right)^2 
 e^{p\alpha u} \psi^p \,dx\\
&\quad + (p-1)\int_{\mathbf{R}^N} w(x)|\nabla u|^{p-1} \frac{p^2\alpha}{2} 
 e^{p\alpha u} \psi^{p-1} |\nabla\psi| \,dx\\
&\quad + (p-1)\int_{\mathbf{R}^N} w(x)|\nabla u|^{p-2} \left(\frac{p}{2}\right)^2 
 e^{p\alpha u} \psi^{p-2} |\nabla\psi|^2 \,dx.
\end{align*}
Using Young's inequality we estimate the last two terms
\begin{align*}
&(p-1)\int_{\mathbf{R}^N} w(x)|\nabla u|^{p-1} \frac{p^2\alpha}{2} 
 e^{p\alpha u} \psi^{p-1} |\nabla\psi| \,dx\\
&\leq \int_{\mathbf{R}^N} \frac{\epsilon}{2}
 \left(w^{\frac{p-1}{p}}|\nabla u|^{p-1} e^{(p-1)\alpha u} \psi^{p-1}\right)^{\frac{p}{p-1}} 
 + C_\epsilon\left(w^{\frac{1}{p}}e^{\alpha u}|\nabla\psi|\right)^p \,dx\\
&= \frac{\epsilon}{2}\int_{\mathbf{R}^N} w(x)|\nabla u|^p e^{p\alpha u}
 \psi^p \,dx + C_\epsilon\int_{\mathbf{R}^N} w(x)e^{p\alpha u} |\nabla\psi|^p \,dx,
\end{align*}
and
\begin{align*}
&(p-1)\int_{\mathbf{R}^N} w(x)|\nabla u|^{p-2} \left(\frac{p}{2}\right)^2 
 e^{p\alpha u} \psi^{p-2} |\nabla\psi|^2 \,dx\\
&\leq \int_{\mathbf{R}^N} \frac{\epsilon}{2}\left(w^{\frac{p-2}{p}}|\nabla u|^{p-2} 
 e^{(p-2)\alpha u} \psi^{p-2}\right)^{\frac{p}{p-2}} 
 + C_\epsilon\left(w^{\frac{2}{p}}e^{2\alpha u}|\nabla\psi|^2\right)^{\frac{p}{2}} \,dx\\
&= \frac{\epsilon}{2}\int_{\mathbf{R}^N} w(x)|\nabla u|^p e^{p\alpha u} \psi^p \,dx 
 + C_\epsilon\int_{\mathbf{R}^N} w(x)e^{p\alpha u} |\nabla\psi|^p \,dx.
\end{align*}
Using these two estimates in the previous one, we obtain after using (\ref{s1}) 
\begin{align*}
&\int_{\mathbf{R}^N} e^{(p\alpha + 1) u} \psi^p \,dx \\
&\leq \Big(\frac{(p-1)p^2\alpha^2}{4}+\epsilon\Big)
 \int_{\mathbf{R}^N} w(x)|\nabla u|^p e^{p\alpha u} \psi^p \,dx 
 + C_\epsilon\int_{\mathbf{R}^N} w(x)e^{p\alpha u} |\nabla\psi|^p \,dx\\
&\leq \Big(\frac{(p-1)p^2\alpha^2}{4}+\epsilon\Big)
 \frac{1}{p\alpha-\epsilon}\int_{\mathbf{R}^N} e^{(p\alpha + 1) u} \psi^p \,dx
 + C_\epsilon\int_{\mathbf{R}^N} w(x)e^{p\alpha u} |\nabla\psi|^p \,dx.
\end{align*}
Define $\gamma_\epsilon=1-\Big(\frac{(p-1)p^2\alpha^2}{4}+\epsilon\Big)
 \frac{1}{p\alpha-\epsilon}.$
Note that 
$\lim_{\epsilon\to 0}\gamma_\epsilon = 1 - \frac{\alpha p(p-1)}{4} > 0$ for $\alpha \in \left(0,\frac{4}{p(p-1)}\right)$.\\ Hence we can choose some $\epsilon\in(0, 1)$ depending on $p$ and $\alpha$ such that $\gamma_\epsilon > 0$.
Hence we get,
\begin{equation}\label{s2}
\int_{\mathbf{R}^N} e^{(p\alpha + 1) u} \psi^p \,dx 
\leq C\int_{\mathbf{R}^N} w(x)e^{p\alpha u} |\nabla\psi|^p \,dx.
\end{equation}
where $C$ is a constant depending on the chosen $\epsilon\in(0,1)$.
Next we choose some $m$ such that $p\alpha+1=m$ and apply (\ref{s2}) for $\psi=\eta^m$ to obtain 
\begin{align*}
\int_{\mathbf{R}^N} e^{(p\alpha + 1) u} \eta^{pm} \,dx 
&\leq C\int_{\mathbf{R}^N} w(x)e^{p\alpha u} \eta^{p(m-1)} |\nabla\eta|^p \,dx\\
&\leq \int_{\mathbf{R}^N} \epsilon 
 \Big(e^{p\alpha u} \eta^{p(m-1)}\Big)^{\frac{p\alpha+1}{p\alpha}} 
 + C_\epsilon (w|\nabla\eta|^p)^{p\alpha + 1} \,dx\\
&\leq \epsilon \int_{\mathbf{R}^N} e^{(p\alpha+1)u} \eta^{pm} \,dx 
 + C_\epsilon \int_{\mathbf{R}^N} w^{p\alpha+1}|\nabla\eta|^{p(p\alpha + 1)} \,dx.
\end{align*}
Choosing $\eta_R\in C_c^1(\mathbf{R}^N)$ satisfying 
$0\leq\eta_R\leq 1$ in $\mathbf{R}^N$, $\eta_R=1$ in $B_R(0)$ and 
$\eta_R=0$ in $\mathbf{R}^N \setminus B_{2R}(0)$ with $|\nabla\eta|\leq\frac{C}{R}$ for $C>1$ we obtain after using the assumption on $w$
\begin{equation}\label{con}
\int_{B_R(0)} e^{(p\alpha + 1) u} \,dx\leq C R^{-p(p\alpha + 1)}\int_{B_{2R}(0)}w^{p\alpha+1} dx
\end{equation}
Now since, $A_R=\mathcal{O}(R^{\mu})$ we have for sufficiently large $R$,
\begin{equation*}
\int_{B_R(0)} e^{(p\alpha + 1) u} \,dx\leq C R^{\big[\mu (\frac{p-1}{p+3})-p\big](p\alpha + 1)}dx
\end{equation*}
which tends to $0$ as $R\to\infty$. Hence arriving at a contradiction. 
\end{proof}

\begin{proof}[Proof of Theorem \ref{radial}]
Following the exact proof of Theorem \ref{nonexistence theorem} for $w(x)\leq C|x|^\theta$ we have from (\ref{con}),
\begin{align*}
\int_{B_R(0)}e^{p(p\alpha+1)u}dx
&\leq CR^{-p(p\alpha+1)}\int_{B_{2R}(0)}(w(x))^{p\alpha+1}r^{N-1}dr\\
&\leq CR^{-p(p\alpha+1)}\int^{2R}_0 r^{\theta(p\alpha+1)+N-1}dr\\
&\leq CR^{(\theta-p)(p\alpha+1)+N}
\end{align*}
which implies that $R^{(\theta-p)(p\alpha+1)+N}\to 0\;\;\mbox{as}\;\;R\to\infty$, given that 
$$
-\frac{N}{p\alpha+1}<\theta<p-\frac{N}{p\alpha+1}
$$
for any $\alpha\in (0,\frac{4}{p(p-1)})$. This proves the first part of the Theorem.\\
Note that when $N<\frac{p(p+3)}{p-1}$, there exist an $\alpha_0\in(0,\frac{4}{p(p-1)})$ such that $N-p(p\alpha_0+1)<0$, which implies $p-\frac{N}{p\alpha_0+1}>0$. \\Hence the interval $\mathcal{I}_{{\alpha}_0}:=(-\frac{N}{p\alpha_0+1},p-\frac{N}{p\alpha_0+1})$ contains both positive and negative values of $\theta$. This along with the fact that for any $\theta$ in $\mathcal{I}_{\alpha}$ and $\alpha\in (0,\frac{4}{p(p-1)})$ we have, $$R^{(\theta-p)(p\alpha+1)+N}\to 0\;\;\mbox{as}\;\;N\to\infty$$ shows that there are $\theta$, both positive and negative for which the problem $(\ref{mp})_e$ does not admit a stable solution.

Again when $N\geq\frac{p(p+3)}{p-1}$, we have $p-\frac{N(p-1)}{p+3}\leq 0$. Note that for any $\alpha\in(0,\frac{4}{p(p-1)})$, we have the strict inequality
$$
p-\frac{N}{p\alpha+1}<p-\frac{N(p-1)}{p+3}.
$$ 
Hence $\mathcal{I}_{\alpha}\subset (-\infty,0)$ when $N\geq\frac{p(p+3)}{p-1}$ and so all exponent $\theta$ must be negative for the non-existence to hold.

\end{proof}

\begin{proof}[Proof of Theorem \ref{Main}]\label{sec2}
Suppose $u\in C^{1,\zeta}(\mathbb{R}^N)$ be a positive stable solution of equation $(\ref{mp})_s$ for $f(x)=-x^{-\delta}$ and $\psi\in C_c^{1}(\mathbb{R}^N)$. \\
\textbf{Step 1.} Choosing $\varphi=u^{-\alpha}\psi^p$, ($\alpha>0$ to be chosen later) as a test function in the weak form $(\ref{ws1})_s$, since
$$
\nabla\varphi=-\alpha u^{-\alpha-1}\psi^p\nabla u+p\psi^{p-1}u^{-\alpha}\nabla\psi
$$
we obtain 
\begin{multline}\label{e5}
\alpha\int_{\mathbb{R}^N}w(x)u^{-\alpha-1}\psi^p|\nabla u|^p dx
\leq p\int_{\mathbb{R}^N}w(x)u^{-\alpha}\psi^{p-1}|\nabla u|^{p-1}|\nabla \psi|dx\\
+\int_{\mathbb{R}^N}u^{-\delta-\alpha}\psi^p dx
\end{multline}
Now using the Young's inequality for $\epsilon\in(0,\alpha)$, we obtain 
\begin{multline*}
p\int_{\mathbb{R}^N}w(x)u^{-\alpha}\psi^{p-1}|\nabla u|^{p-1}|\nabla \psi|dx\\
=p\int_{\mathbb{R}^N}(w^\frac{1}{p'}u^{\frac{-\alpha-1}{p'}}|\nabla u|^{p-1}\psi^{p-1})(w^\frac{1}{p}u^\frac{p-\alpha-1}{p}|\nabla \psi|)dx\\
\leq \{\epsilon\int_{\mathbb{R}^N}w(x)u^{-\alpha-1}|\nabla u|^p\psi^p dx\\+C_{\epsilon}\int_{\mathbb{R}^N}w(x)u^{p-\alpha-1}|\nabla\psi|^p\}dx
\end{multline*}
Plugging this estimate in $(\ref{e5})$, we get 
\begin{multline}\label{e7}
(\alpha-\epsilon)\int_{\mathbb{R}^N}w(x)u^{-\alpha-1}|\nabla u|^p\psi^p dx\leq C_{\epsilon}\int_{\mathbb{R}^N}w(x)u^{p-\alpha-1}|\nabla \psi|^p dx\\
+\int_{\mathbb{R}^N}u^{-\delta-\alpha}\psi^p dx
\end{multline}
\textbf{Step 2.} Choosing $\varphi=u^{-\beta-\frac{p}{2}+1}\psi^{\frac{p}{2}}$, ($\beta>0$ to be chosen later) as a test function in the stability equation $(\ref{ss})_s$, since
$$
\nabla\varphi=(-\beta-\frac{p}{2}+1)u^{-\beta-\frac{p}{2}}\psi^{\frac{p}{2}}\nabla u+\frac{p}{2}\psi^{\frac{p-2}{2}}u^{-\beta-\frac{p}{2}+1}\nabla\psi
$$
we obtain
\begin{multline*}
\delta\int_{\mathbb{R}^N}u^{-2\beta-\delta-p+1}\psi^p dx
\leq (p-1)(-\beta-\frac{p}{2}+1)^2\int_{\mathbb{R}^N} w(x)u^{-2\beta-p}|\nabla u|^p\psi^p dx\\
+(p-1)\frac{p^2}{4}\int_{\mathbb{R}^N}w(x)u^{-2\beta-p+2}\psi^{p-2}|\nabla u|^{p-2}|\nabla\psi|^2 dx\\+
p(p-1)(-\beta-\frac{p}{2}+1)\int_{\mathbb{R}^N}w(x)u^{-2\beta-p+1}\psi^{p-1}|\nabla u|^{p-1}|\nabla\psi|dx
=A+B+C
\end{multline*}
where 
$$
A=(p-1)(-\beta-\frac{p}{2}+1)^2\int_{\mathbb{R}^N} w(x)u^{-2\beta-p}|\nabla u|^p\psi^p dx
$$
$$
B=(p-1)\frac{p^2}{4}\int_{\mathbb{R}^N}w(x)u^{-2\beta-p+2}\psi^{p-2}|\nabla u|^{p-2}|\nabla\psi|^2 dx
$$
and 
$$
C=p(p-1)(-\beta-\frac{p}{2}+1)\int_{\mathbb{R}^N}w(x)u^{-2\beta-p+1}\psi^{p-1}|\nabla u|^{p-1}|\nabla\psi|dx.
$$
Therefore we have 
\begin{equation}\label{e8}
\delta\int_{\mathbb{R}^N}u^{-2\beta-\delta-p+1}\psi^p dx\leq A+B+C
\end{equation}
Now, using the exponents $\frac{p}{p-2}$ and $\frac{p}{2}$ in the Young's inequality, we have
\begin{align*}
B&=(p-1)\frac{p^2}{4}\int_{\mathbb{R}^N}w(x)u^{-2\beta-p+2}\psi^{p-2}|\nabla u|^{p-2}|\nabla\psi|^2 dx\\
&=(p-1)\frac{p^2}{4}\int_{\mathbb{R}^N}(w^\frac{p-2}{p}|\nabla u|^{p-2}\psi^{p-2}u^\frac{(-2\beta-p)(p-2)}{p})(w^\frac{2}{p}u^\frac{2(p-2\beta-p)}{p}|\nabla \psi|^2)dx\\
&\leq\{\frac{\epsilon}{2}\int_{\mathbb{R}^N}w(x)|\nabla u|^p\psi^pu^{-2\beta-p}dx+C_{\epsilon}\int_{\mathbb{R}^N}w(x)u^{-2\beta}|\nabla \psi|^p dx\}
\end{align*}
Also using the exponents $p'=\frac{p}{p-1}$ and $p$ in the Young's inequality, we obtain
\begin{align*}
C&=p(p-1)(-\beta-\frac{p}{2}+1)\int_{\mathbb{R}^N}w(x)u^{-2\beta-p+1}\psi^{p-1}|\nabla u|^{p-1}|\nabla\psi|dx\\
&=p(p-1)(-\beta-\frac{p}{2}+1)\int_{\mathbb{R}^N}(w^\frac{1}{p'}|\nabla u|^{p-1}\psi^{p-1}u^\frac{-2\beta-p}{p'})(w^\frac{1}{p}u^\frac{p-2\beta-p}{p}|\nabla \psi|)dx\\
&\leq\{\frac{\epsilon}{2}\int_{\mathbb{R}^N}w(x)|\nabla u|^p\psi^pu^{-2\beta-p}dx+C_{\epsilon}\int_{\mathbb{R}^N}w(x)u^{-2\beta}|\nabla \psi|^p dx\}
\end{align*}
Choosing $\alpha=2\beta+p-1>0$ in equation (\ref{e7}) we get 
\begin{multline}\label{e9}
\int_{\mathbb{R}^N}w(x)u^{-2\beta-p}|\nabla u|^p\psi^p dx\leq\frac{1}{2\beta+p-\epsilon-1}\\
\{C_\epsilon\int_{\mathbb{R}^N}w(x)u^{-2\beta}|\nabla \psi|^p dx+\int_{\mathbb{R}^N}u^{-2\beta-\delta-p+1}\psi^p dx\}
\end{multline}
Using the inequality $(\ref{e9})$ and the above estimates on $B$ and $C$ in $(\ref{e8})$, we get
\begin{align*}
\delta\int_{\mathbb{R}^N}u^{-2\beta-\delta-p+1}\psi^p dx
\leq\frac{(p-1)(-\beta-\frac{p}{2}+1)^2+\epsilon}{(2\beta+p-\epsilon-1)}\int_{\mathbb{R}^N}u^{-2\beta-\delta-p+1}\psi^p dx+\\C_\epsilon (\frac{(p-1)(-\beta-\frac{p}{2}+1)^2+\epsilon}{(2\beta+p-\epsilon-1)}+2)\int_{\mathbb{R}^N} w(x)u^{-2\beta}|\nabla \psi|^p dx.
\end{align*}
We define $\eta_{\epsilon}=\delta-\frac{(p-1)(-\beta-\frac{p}{2}+1)^2+\epsilon}{(2\beta+p-\epsilon-1)}$\\
Therefore, we have $\lim\limits_{\epsilon\to 0}\eta_{\epsilon}=\delta-\frac{(p-1)(-\beta-\frac{p}{2}+1)^2}{(2\beta+p-1)}>0$ for every $\beta\in(0,l)$.\\
Therefore we can choose an $\epsilon\in(0,1)$ depending on $p$ and $\beta$ such that $\eta_{\epsilon}>0$.
Hence we get
\begin{equation}\label{e10}
\int_{\mathbb{R}^N}u^{-2\beta-\delta-p+1}\psi^p dx\leq C\int_{\mathbb{R}^N}w(x)u^{-2\beta}|\nabla \psi|^p dx.
\end{equation}
Replacing $\psi$ by $\psi^\frac{2\beta+\delta+p-1}{p}$ we get
\begin{align*}
\int_{\mathbb{R}^N}(\frac{\psi}{u})^{2\beta+\delta+p-1}dx
&\leq C(\frac{2\beta+\delta+p-1}{p})^p\int_{\mathbb{R}^N}w(x)u^{-2\beta}\psi^{2\beta+\delta-1}|\nabla\psi|^p dx\\
&=C\int_{\mathbb{R}^N}(\frac{\psi}{u})^{2\beta}(w(x)\psi^{\delta-1}|\nabla\psi|^p) dx
\end{align*}
Choosing the exponents $\gamma=\frac{2\beta+\delta+p-1}{2\beta}$ and ${\gamma}'=\frac{\gamma}{\gamma-1}=\frac{2\beta+\delta+p-1}{\delta+p-1}$ in the Young's inequality we get,
\begin{align*}
\int_{\mathbb{R}^N}(\frac{\psi}{u})^{2\beta+\delta+p-1}dx
&\leq\{\epsilon\int_{\mathbb{R}^N}(\frac{\psi}{u})^{2\beta+\delta+p-1}dx+C_{\epsilon}\int_{\mathbb{R}^N}w^{\gamma'}\psi^{(\delta-1){\gamma}'}|\nabla\psi|^{p{\gamma}'}dx\}
\end{align*}
Therefore we get the inequality
\begin{equation}\label{Cac}
\int_{\mathbb{R}^N}(\frac{\psi}{u})^{2\beta+\delta+p-1}dx
\leq C\int_{\mathbb{R}^N}w^\frac{2\beta+\delta+p-1}{\delta+p-1}(\psi^{\frac{\delta-1}{p}}|\nabla\psi|)^{\frac{p(2\beta+\delta+p-1)}{\delta+p-1}}dx.
\end{equation}

Choose $\psi_R\in{C}_{c}^{1}(\mathbb{R}^N)$ satisfying $0\leq\psi_R\leq {1}$ in $\mathbb{R}^N$, $\psi_R=1$ in $B_R(0)$ and $\psi_R=0$ in $\mathbb{R}^N\setminus B_{2R}(0)$ with $|\nabla\psi_R|\leq\frac{C}{R}$ for $C>1$,
\begin{equation}\label{Final}
\int_{B_R(0)}(\frac{1}{u})^{2\beta+\delta+p-1}dx
\leq C{R}^{-\frac{p(2\beta+\delta+p-1)}{\delta+p-1}}\int_{B_{2R}(0)}w^\frac{2\beta+\delta+p-1}{\delta+p-1}dx.
\end{equation}
Since $Q_R=\mathcal O(R^\tau)$, letting $R\to\infty$ we have
\begin{equation}\nonumber
\int_{B_R(0)}(\frac{1}{u})^{2\beta+\delta+p-1}dx \leq C R^{\tau (\frac{\delta+p-1+2\beta}{\delta+p-1+2l})-p(\frac{\delta+p-1+2\beta}{\delta+p-1})}
\end{equation}
Now since, $\tau<pk$ we have, $$R^{\tau (\frac{\delta+p-1+2\beta}{\delta+p-1+2l})-p(\frac{\delta+p-1+2\beta}{\delta+p-1})}\to 0$$ as $R\to\infty$, where $l:=\frac{2\delta}{p-1}-\frac{p-1}{2}$.

This implies $\displaystyle\int_{\mathbb{R}^N}(\frac{1}{u})^{2\beta+\delta+p-1}dx=0$ and so we arrive at a contradiction. 
\end{proof}

\begin{proof}[Proof of Theorem \ref{rad}]
Following the proof of Theorem \ref{Main} for $w(x)\leq C|x|^\theta$ we have from (\ref{Final}),
\begin{align}
\displaystyle\int_{B_R(0)}(\frac{1}{u})^{2\beta+\delta+p-1}dx&\leq C{R}^{-\frac{p(2\beta+\delta+p-1)}{\delta+p-1}}\displaystyle\int_{B_{2R}(0)}w^\frac{2\beta+\delta+p-1}{\delta+p-1}dx\nonumber\\
&\leq C{R}^{-\frac{p(2\beta+\delta+p-1)}{\delta+p-1}} \displaystyle\int_{B_{2R}(0)}r^{\frac{\theta(2\beta+\delta+p-1)}{\delta+p-1}}r^{N-1}dr\nonumber\\
&=C R^{\frac{(\theta-p)(2\beta+\delta+p-1)}{\delta+p-1}+N}\label{imp}
\end{align}
Now given, $\theta\in (-\frac{N(\delta+p-1)}{2\beta+\delta+p-1},p-\frac{N(\delta+p-1)}{2\beta+\delta+p-1})$ one has, $\frac{(\theta-p)(2\beta+\delta+p-1)}{\delta+p-1}+N<0$. Hence  we have from (\ref{imp}) that $\displaystyle\int_{\mathbb{R}^N}(\frac{1}{u})^{2\beta+\delta+p-1}dx=0$ which is a contradiction to the fact that $u$ is a stable solution.

When $N\geq pk$ then one has, $$p-\frac{N(\delta+p-1)}{2\beta+\delta+p-1}<p-\frac{N(\delta+p-1)}{2l+\delta+p-1}\leq 0\;\;\mbox{for all}\;\;\beta\in\mathcal{B}$$ where 
$l:=\frac{2\delta}{p-1}-\frac{p-1}{2}$.\\

Therefore one has $\mathcal{K}_\beta\subset(-\infty,0)$ for all $\beta\in\mathcal{B}$.\\
For $N<pk$, there exist $\beta_{0}\in(0,l)$ such that $N-{\frac{p(2\beta_{0}+\delta+p-1)}{\delta+p-1}}<0$ 
which implies,
$$
p-\frac{N(\delta+p-1)}{2\beta+\delta+p-1}>0.
$$
Hence $0\in\mathcal{K}_{\beta_{0}}.$
\end{proof}

\section*{Acknowledgement:}
The first author and second author are supported by Inspire Faculty Award DST-MATH 2013-029
 and by NBHM Fellowship No: 2-39(2)-2014 (NBHM-RD-II-8020-June 26, 2014) respectively.

Kaushik Bal\\
E-mail: kaushik@iitk.ac.in\\
Department of Mathematics and Statistics.\\
Indian Institute of Technology, Kanpur.\\
UP-208016, India.\\

Prashanta Garain\\
E-mail: pgarain@iitk.ac.in\\
Department of Mathematics and Statistics.\\
Indian Institute of Technology, Kanpur.\\
UP-208016, India.

\begin{thebibliography}{0}

\bibitem{Ca} Cabr\'e, Xavier
\emph{Regularity of minimizers of semilinear elliptic problems up to dimension 4.}
Comm. Pure Appl. Math. 63 (2010), no. 10, 1362–-1380.

\bibitem{Ca1} Cabr\'e, Xavier; Capella, Antonio
\emph{Regularity of radial minimizers and extremal solutions of semilinear elliptic equations.}
J. Funct. Anal. 238 (2006), no. 2, 709–-733. 

\bibitem{Ca3}Cabr\'e, Xavier
\emph{Boundedness of stable solutions to semilinear elliptic equations: a survey.}
Adv. Nonlinear Stud. 17 (2017), no. 2, 355–-368.

\bibitem{ChSoYa} Chen, C; Song, H; Yang, H
\emph{Liouville type theorems for stable solutions of p-Laplace equation in $\mathbb R^N$.}
Nonlinear Anal. 160 (2017), 44–-52

\bibitem{DrKu} Dr\'abek, P; Kufner, A; Nicolosi, F
\emph{Quasilinear elliptic equations with degenerations and singularities.}
De Gruyter Series in Nonlinear Analysis and Applications, Berlin, 1997. xii+219 pp

\bibitem{Ne} Nedev, G
\emph{Regularity of the extremal solution of semilinear elliptic equations.}
C. R. Acad. Sci. Paris Sér. I Math. 330 (2000), no. 11, 997–-1002.

              
\bibitem{Fa} Farina, A
\emph{Stable solutions of {$-\Delta u=e^u$} on {$\mathbb R^N$}},
C. R. Math. Acad. Sci. Paris, Vol-345, 2007(2) 63--66.

\bibitem{Fa1}Farina, A
\emph{Liouville-type results for solutions of $−\Delta u=|u|^{p−1}u$ on unbounded domains of $\mathbb R^N$}
C. R. Math. Acad. Sci. Paris 341 (2005), no. 7, 415–-418. 

\bibitem{Fa2} Farina, A
\emph{On the classification of solutions of the Lane-Emden equation on unbounded domains of $\mathbb R^N$. }
J. Math. Pures Appl. (9) 87 (2007), no. 5, 537–561. 

\bibitem{DuGuo} Du, Yihong; Guo, Zongming
\emph{Positive solutions of an elliptic equation with negative exponent: stability and critical power.} 
J. Differential Equations 246 (2009), no. 6, 2387–-2414.

\bibitem{GuoMei} Guo, Zong Ming; Mei, Lin Feng
\emph{Liouville type results for a p-Laplace equation with negative exponent.}
Acta Math. Sin 32 (2016), no. 12, 1515–-1540. 

\bibitem{Ph} Le, P
\emph{Nonexistence of stable solutions to p-Laplace equations with exponential nonlinearities.} 
Electron. J. Differential Equations 2016, Paper No. 326, 5 pp. 

\bibitem{HuYe} Huang, X; Ye, D
Existence of stable solutions to {$(-\Delta)^mu=e^u$} in
              {$\mathbb{R}^N$} with {$m\geq 3$} and {$N>2m$}J. Differential Equations 260 (2016), no. 8, 6493--6503.

\bibitem{JWei} Ma, Li; Wei, J. C. 
Properties of positive solutions to an elliptic equation with negative exponent. J. Funct. Anal. 254 (2008), no. 4, 1058-1087
\end{thebibliography}
\end{document}